\documentclass[12pt]{article}
\usepackage[psamsfonts]{amsfonts}
\usepackage{amsmath,amssymb,amsthm,dsfont,enumerate}
\usepackage{graphicx}               

\providecommand{\keywords}[1]{\paragraph{Keywords} #1}
\providecommand{\msc}[1]{\paragraph{Mathematics Subject Classification} #1}

\newtheorem{theorem}{Theorem}[section]

\newtheorem{definition}[theorem]{Definition}
\newtheorem{lemma}[theorem]{Lemma}

\newtheorem{proposition}[theorem]{Proposition}

\newtheorem*{remark}{Remark}

\newcommand{\eq}{\begin{equation}}
\newcommand{\en}{\end{equation}}
\newcommand{\nn}{\nonumber}

\newcommand{\prob}{\mathbb P}
\newcommand{\expec}{\mathbb E}

\numberwithin{equation}{section}
\DeclareMathOperator*{\argmin}{arg\,min}
\DeclareMathOperator*{\argmax}{arg\,max}

\title  {Metastability of the Ising model on random regular graphs at zero temperature}

\author
{
 Sander Dommers
 \footnote{
   Universit\`a di Bologna, Dipartimento di Matematica,
   Piazza di Porta San Donato 5, 40126 Bologna, BO, Italy.
   E-mail: {\tt sander.dommers@unibo.it}
 }
}

\date{\today}

\begin{document}

\maketitle

\begin{abstract}
We study the metastability of the ferromagnetic Ising model on a random $r$-regular graph in the zero temperature limit. We prove that in the presence of a small positive external field the time that it takes to go from the all minus state to the all plus state behaves like ${\exp(\beta (r/2+ \mathcal{O}(\sqrt{r}))n)}$ when the inverse temperature $\beta\rightarrow\infty$ and the number of vertices $n$ is large enough but fixed. The proof is based on the so-called pathwise approach and bounds on the isoperimetric number of random regular graphs.
\end{abstract}

\keywords{Metastability, Ising model, random graphs, pathwise approach}

\msc{60K35, 82C20}

\section{Introduction}

One of the most important models in the study of phase transitions is the Ising model. In this model, to every vertex a spin is assigned that can have value $+1$ or $-1$. These spins interact with each other and have the tendency to align: the spins of vertices that are connected by an edge tend to get the same value. This model was originally suggested by Lenz to his student Ising to study ferromagnetism~\cite{Isi24}, but later also became a model to study cooperative behavior. For a historic account of this model see~\cite{Nis05,Nis09,Nis11}.

In this paper, we focus on the dynamical properties of this model, when the system evolves according to Glauber dynamics: at every time step a vertex is selected uniformly at random, if flipping this spin would result in a lower energy this happens with probability one, whereas if the energy would increase this happens only with an exponentially small probability. This is also called the stochastic Ising model, for example in~\cite{Lig85}. In particular, we focus on {\em metastability} in this model. When there is a positive external field present, the stable state when the temperature tends to zero is the one where all vertices have spin $+1$. If, however, we start with a system where all spins are equal to $-1$ and let the system evolve, then it will take a very long time before the system reaches this stable state due to the strong interaction between the spins. Hence, on a short time scale $\boxminus$ seems stable. This is what we call a {\em metastable state} and the time it takes for the system to reach the stable state is called the {\em metastable time}.

The three main problems in the study of metastability are to estimate the metastable time, to determine the shape of the critical droplet that the system visits with probability tending to one as the temperature tends to zero, and to determine the tube of typical trajectories or paths that the system visits with probability tending to one as the temperature tends to zero during the transition from all minuses to all pluses.

Several methods have been developed to study these problems. The first method is called the {\em pathwise approach}, which is also the method used in this paper. This method was first introduced in~\cite{CasGalOliVar84} and further developed in, e.g.,~\cite{OliSco95, OliVar05}. In this method the three main problems are addressed simultaneously: by obtaining large deviation results on the tube of typical trajectories one obtains the shape of the critical droplet and also exponential asymptotics on the metastable time. In~\cite{ManNarOliSco04}, the control of the tube of typical trajectories is decoupled from the control of the metastable time.

For the Ising model, the pathwise approach has been applied to the large enough but finite two-dimensional torus, in the setting where the temperature tends to zero, in~\cite{NevSch91}. There it is shown that, if we call the metastable time $\tau$ and the inverse temperature $\beta$ (more precise definitions will be given in Section~\ref{sec-defs}), that $1/\beta \log \tau \to \Gamma$ in probability as $\beta\to\infty$ for some explicit value $\Gamma$. This $\Gamma$ is called the energy barrier and is the energy difference between the configuration where a critical droplet of $+1$'s has been formed and the configuration where all spins are $-1$. This critical droplet is described exactly, it is shown that such a configuration is visited with probability tending to one as $\beta\to\infty$ and hence also $\Gamma$ is known explicitly. The value of $\Gamma$ depends on the value of the external field, but not on the (finite) volume of the torus. In~\cite{Sch92}, also the tube of typical trajectories or paths that the system visits with probability tending to one as the temperature tends to zero is described. In~\cite{BenCer96}, these three problems where studied and similar results where obtained for the three-dimensional torus.

The pathwise approach has also been applied in~\cite{KotOli93} to a two dimensional anisotropic Ising model where the interaction strengths between spins in the horizontal and vertical direction are different. In~\cite{KotOli94}, a two-dimensional Ising model is studied where also next nearest neighbor interactions are present. In these two papers results analogous to those in~\cite{NevSch91} are given, although the anisotropy causes the formation and shape of the critical droplet to be different.

When the system evolves according to Glauber dynamics the number of sites with spin $+1$ is not conserved. A related model that is conservative is called Kawasaki dynamics. In this model $+$ spins can enter and leave the system from the boundary. Also in this case one can look in the zero temperature limit to the formation of a critical droplet that has to be formed, before the system quickly relaxes to a state with all $+$ spins. The pathwise approach for this model has been used in two and three dimension in \cite{HolOliSco00} and \cite{HolNarOliSco03}, respectively.

A second method is called the {\em potential theoretic approach} which was developed in~\cite{BovEckGayKle01,BovEckGayKle02}. In this method, instead of controlling the tube of typical trajectories, one has to compute capacities between sets of (meta)stable states, which can be done using variational principles. For the Ising model, this method has for example been used in~\cite{BovMan02} to obtain sharper results on the mean metastable time for the two- and three-dimensional torus. There it is shown that $\expec[\tau]=K e^{\beta \Gamma}(1+o(1))$ as $\beta\to\infty$ for some explicit constants $K$ and $\Gamma$ depending on the dimension and the size of the external field.

Recently, also a third method, called the {\em martingale approach}, has been developed in~\cite{BelLan10,BelLan14}. This method allows to show that the process with a properly rescaled time converges to a Markov process on the set of (meta)stable states, where the stable state is absorbing. In~\cite{BelLan11}, this method is illustrated for the two-dimensional Ising model in the zero temperature limit. 

When interpreting the Ising model as a model for cooperative behavior, studying it in a setting where the graph is a lattice makes less sense. Hence, in recent years there has been a large interest in studying the Ising model, and other models, on random graphs, which are itself models for complex networks, see for example~\cite{New03} for an overview. We focus on one of the simplest random graphs, called the {\em random regular graph}, where all vertices have the same degree and the graph is chosen uniformly at random among all graphs with this property.

The Ising model on a random graph was first studied rigorously in~\cite{SanGue08}, where the high temperature and zero temperature solution at equilibrium where obtained for the Erd\H{o}s-R\'enyi random graph using interpolation techniques. Later, in~\cite{DemMon10} the equilibrium solution at all temperatures was obtained for random graphs that locally behave like a branching process. Such graphs include the Erd\H{o}s-R\'enyi random graph and random regular graphs, but also other random graphs with an uncorrelated degree sequence with finite variance degrees. The latter condition was relaxed in~\cite{DomGiaHof10} to random graphs where the degrees have strongly finite mean. The weak limit of the Ising measure in zero field is studied in more detail in~\cite{MonMosSly12} and~\cite{BasDem12}. In~\cite{DomGiaHof14}, the critical exponents for this model have been computed.

About the dynamical properties of the Ising model on random graphs, however, not much is known. In~\cite{MosSly09,MosSly13} the mixing time for the the Erd\H{o}s-R\'enyi random graph and random regular graphs was computed in the high temperature regime. In~\cite{LubSly14}, a cut-off phenomenon was proved for high enough temperature. To the best of our knowledge, on the low temperature behavior only some simulation results have appeared in the literature~\cite{CheEtAl13,CheSheHouXin11, SheCheYeHou13}.

In this paper, we present results on the dynamical properties of the Ising model on random $r$-regular graphs in the the zero temperature limit. We give bounds on the metastable time and show that it is exponential in the inverse temperature $\beta$, the number of vertices in the graph $n$ and the degrees of the vertices $r$, i.e., we show that there exist constants $0<C_1<\infty$ and $C_2<\infty$ such that $e^{\beta(r/2-C_1\sqrt{r})n}<\tau<e^{\beta(r/2+C_2\sqrt{r})n}$ with probability tending to one as $\beta\to\infty$. For the exact result, see Theorem~\ref{thm-metastabletime} below. The exponential growth of the metastable time with $n$ is an important difference from the Ising model on finite dimensional lattices, where the metastable time is uniformly bounded in the size of the graph, see, e.g.,~\cite{NevSch91} for $d=2$ and \cite{BenCer96} for $d=3$.

Whereas the results on the equilibrium solution mentioned above are mainly based on the locally tree-like structure of the random graphs, our results on the metastable time are based on the structure of subsets consisting of a positive fraction of the vertices. We use results by Bollob\'as~\cite{Bol88} and Alon~\cite{Alo97} concerning the {\em isoperimetric number} of random regular graphs, which is a measure of how connected to the rest of the graph a subset consisting of half of the vertices has to be.

We combine these results with the pathwise approach. We use in particular the results in~\cite{ManNarOliSco04}, where, as mentioned, the control of the tube of typical trajectories is decoupled from the control of the metastable time. This is very useful in our setting, since the random structure of the graphs make it very hard to deduce what these tubes look like precisely. Also for the other approaches to metastability much more detailed model dependent information would be necessary. We combine this with the approach used in~\cite{CirNar13}, which simplifies the necessary calculations.


\section{Model definitions and preliminary results}\label{sec-defs}
In this section we give the definitions of random regular graphs and of the Ising model, its dynamics and metastability. We also present several preliminary results on these topics that are used in the paper. 
\subsection{Random regular graphs}
We denote a graph $G_n=(V_n=[n],E_n)$ where $[n]=\{1,\ldots,n\}$ by $G_n$. For given $r\in \mathbb{N}$ and $n > r$ with $n r$ even, we say that $G_n$ is a {\em random $r$-regular graph} if we select $G_n$ uniformly at random from the set of all simple $r$-regular graphs with $n$ vertices.

An event $A_n$ is said to hold {\em with high probability (whp)}, if 
\eq
\lim_{n\rightarrow\infty}\prob[A_n]=1,
\en
where $\prob$ denotes the measure of selecting a random $r$-regular graph.

For a graph $G_n$, the (edge) boundary of a set $A \subseteq [n]$ equals
\eq
\partial_e A =\{(i,j)\in E_n \,|\, i \in A,\, j \notin A \}.
\en

An important quantity of the graphs of interest is the so-called {\em isoperimetric number}, which is defined as follows:
\begin{definition}[Isoperimetric number] For a graph $G_n$, the (edge) {\em isoperimetric number} of $G_n$ equals
\eq
i_e(G_n) = \min_{\stackrel{A \subset [n]}{|A|\leq n/2}} \frac{|\partial_e A|}{|A|},
\en 
where $|A|$ denotes the cardinality of the set $A$.

Furthermore, define
\eq
i_e'(G_n) = \min_{\stackrel{A \subset [n]}{|A| = \lfloor n/2 \rfloor}} \frac{|\partial_e A|}{|A|}.
\en
\end{definition}
Note that it always holds that $i_e(G_n) \leq i_e'(G_n)$. The following lower bounds on the isoperimetric number were proved in~\cite{Bol88} by Bollob\'as:
\begin{proposition}[Lower bounds on isoperimetric number]\label{prop-lb-isoperimetric}
Let $G_n$ be a random $r$-regular graph with $r\geq3$ and let $\zeta\in(0,1)$ be such that
\eq
2^{4/r}<(1-\zeta)^{1-\zeta}(1+\zeta)^{1+\zeta}.
\en
Then, whp,
\eq
i_e(G_n) \geq (1-\zeta)r/2.
\en
In particular, whp,
\eq\label{eq-lb-isoperimetric}
i_e(G_n) \geq \frac{r}{2} - \sqrt{\log 2}\sqrt{r},
\en
and, for $r\geq6$, whp,
\eq
\label{eq-ieg1}
i_e(G_n) >1.
\en
\end{proposition}
For small $r$ better bounds on $i_e(G_n)$ than~\eqref{eq-lb-isoperimetric} can be obtained as observed in~\cite{Bol88}. This can be used to show~\eqref{eq-ieg1} for $r=6$, whereas~\eqref{eq-lb-isoperimetric} is sufficient for $r\geq7$. When~\eqref{eq-ieg1} holds, it should be noted that for all subsets $A \subset [n]$ with $|A|\leq n/2$, there always exists a vertex in $A$ that is connected to at least two vertices in $A^c$ since $|\partial_e A|>|A|$. This property turns out to be useful later.

For $r$ tending to infinity $r/2-\Theta(\sqrt{r})$ turns out to be the correct scaling of the isoperimetric number, since in~\cite{Alo97} Alon derived the following upper bound:
\begin{proposition}[Upper bound on isoperimetric number]\label{prop-ub-isoperimetric}
There exists an absolute constant $C>0$ such that for any $r$-regular graph $G_n$ with $n\geq 40 r^9$ there exists a set $A \subset [n]$ with $|A|=\lfloor n/2\rfloor$, such that
\eq
\frac{|\partial_e A|}{|A|} \leq \frac{r}{2}-C \sqrt{r}.
\en
Hence,
\eq \label{eq-ub-ieprime}
i_e(G_n) \leq i_e'(G_n) \leq \frac{r}{2}-C \sqrt{r}.
\en
\end{proposition}

\subsection{Ising model, dynamics and metastability}
For fixed $n\in\mathbb{N}$, the Ising model on a graph $G_n$ is defined as follows. To each vertex $i\in[n]$ we assign a spin $\sigma_i\in\{-1,+1\}$ and we denote $\sigma=(\sigma_i)_{i\in[n]}$. The Hamiltonian $H(\sigma)$ is then given by 
\eq\label{eq-hamiltonian}
H(\sigma) = -J \sum_{(i,j)\in E_n} \sigma_i \sigma_j - h \sum_{i\in[n]}\sigma_i,
\en
where $J>0$ is the interaction constant and $h\in \mathbb{R}$ is the external magnetic field.
The Boltzmann-Gibbs measure for the Ising model on $G_n$ is then defined as
\eq
\mu_n(\sigma) = \frac{1}{Z_n} e^{-\beta H(\sigma)},
\en
where $\beta \geq 0$ is the inverse temperature and $Z_n$ is the normalization factor, called the partition function, i.e.,
\eq
Z_n = \sum_{\sigma\in\{-1,+1\}^n} e^{-\beta H(\sigma)}.
\en
Without loss of generality, we assume that $J=1$, since this is just a rescaling of $\beta$ and $h$. 

For a set $A\subseteq [n]$, denote by $\sigma^A$ the configuration where
\eq
\sigma^A_i = \left\{\begin{array}{ll} +1, & {\rm if\ }i \in A, \\ -1, & {\rm if\ } i \notin A. \end{array} \right.
\en
We also denote $\boxminus=\sigma^\varnothing$ and $\boxplus=\sigma^{[n]}$, the all minus and all plus configurations, respectively. We often identify the vertex and its spin, e.g., we say that vertex $i$ has a $+$ neighbor if there is a vertex $j$ such that $(i,j)\in E_n$ with $\sigma_j = +1$.

We let the system evolve according to {\em Glauber dynamics} with Metropolis rates. That is, we consider a discrete time Markov chain where the transition probability $c(\sigma^A,\sigma^B)$ from configuration $\sigma^A$ to $\sigma^B$ equals
\eq
c(\sigma^A,\sigma^B) = \left\{ \begin{array}{ll}  \frac1n e^{-\beta[H(\sigma^B)-H(\sigma^A)]^+}, &{\rm if\ } |A \triangle B|=1; \\ 1- \sum_{B: |A \triangle B|=1} \frac1n e^{-\beta[H(\sigma^B)-H(\sigma^A)]^+}, & {\rm if\ } A=B, \\
0,& {\rm otherwise,} \end{array}\right.
\en
where  $A \triangle B$ is the symmetric difference between sets $A$ and $B$, and $[a]^+=\max\{a,0\}$. We denote by $\prob_\eta$ the law of the process starting from configuration $\eta$.

The time at which the process visits the configuration $\sigma$ for the first time if the process starts from $\eta$ is called the {\em hitting time} of $\sigma$ and is denoted by $\tau_\sigma$. When studying metastability, the problem is to find the hitting time of the stable configuration if the system starts in a metastable configuration. We now define what it means for a configuration to be (meta)stable.

The {\em stable state} is the state for which the Hamiltonian is minimal. Throughout the rest of the paper we assume that $h>0$ and fixed, so that it is obvious from~\eqref{eq-hamiltonian} that the stable state is $\boxplus$.

To define metastable states, we need to define the {\em communication height} between two configurations $\sigma$ and $\sigma'$ which is given by
\eq
\Phi(\sigma,\sigma') = \min_{{\rm \omega\ path\ from\ }\sigma{\rm\ to\ }\sigma'} \max_{\sigma''\in \omega} H(\sigma''),
\en
where we say that a sequence of configurations $\omega$ is a path from $\sigma$ to $\sigma'$ if $\omega=(\sigma=\sigma^{A_0}, \sigma^{A_1}, \ldots,\sigma^{A_\ell}=\sigma')$ for some $\ell\geq1$ and $|A_k\triangle A_{k+1}|=1$ for all $0\leq k<\ell$. We then define the {\em stability level} of a configuration $\sigma$ as
\eq
V_\sigma = \min_{\sigma' : H(\sigma')<H(\sigma)} \Phi(\sigma, \sigma') - H(\sigma).
\en
Note that $V_\boxplus=\infty$ since there are no configurations with smaller energy.
The {\em maximal stability level} is defined as
\eq
\Gamma = \max_{\sigma \neq \boxplus} V_\sigma,
\en
and the {\em metastable states} are those configurations $\eta$ such that $V_\eta=\Gamma$.

In~\cite{CirNar13}, Cirillo and Nardi give an easy characterization of the maximal stability level $\Gamma$ by looking at the communication height between the metastable and stable state, and upper bounds on the stability levels of all other states. This gives an easier way to compute $\Gamma$, since it is no longer necessary to compute the stability level of all states exactly. We use a similar approach to give bounds on $\Gamma$ by verifying the following conditions for some $\Gamma_\ell\leq\Gamma_u$:
\begin{description}
\item[Condition (1)]
$\Phi(\boxminus,\boxplus)- H(\boxminus) \geq \Gamma_\ell$;
\item[Condition (2)]
$\Phi(\boxminus,\boxplus)- H(\boxminus) \leq \Gamma_u$;
\item[Condition (3a)]
for all $\sigma \notin \{\boxminus,\boxplus\}$ it holds that $V_\sigma \leq \Gamma_u$;
\item[Condition (3b)]
for all $\sigma \notin \{\boxminus,\boxplus\}$ it holds that $V_\sigma < \Gamma_\ell$.
\end{description}
Under these conditions we can get the following results on the metastable time:
\begin{proposition}[Metastable time]\label{prop-metastable-time}
If Conditions {\bf (1)}, {\bf (2)} and {\bf (3a)} hold then, for all $\varepsilon>0$,
\eq
\lim_{\beta\rightarrow\infty} \prob_{\boxminus}[e^{\beta(\Gamma_\ell-\varepsilon)}< \tau_{\boxplus} < e^{\beta(\Gamma_u+\varepsilon)}]=1.
\en
If $\eta$ is a metastable state, then the same holds with $\boxminus$ replaced by $\eta$.
If also Condition {\bf (3b)} holds, then $\boxminus$ is the unique metastable state.
\end{proposition}
This proposition can be proved by showing that $\Gamma_\ell\leq\Gamma\leq\Gamma_u$ using the methodology of~\cite{CirNar13}, where the same conditions are used with $\Gamma_\ell=\Gamma_u$, and combining this with  results on the metastable time in~\cite{ManNarOliSco04}. We give this prove in Appendix~\ref{apx-metastabletime}.


\section{Main result and discussion}
We can now present our main result, which gives bounds on the metastable time:
\begin{theorem}[Metastable time for random $r$-regular graphs]\label{thm-metastabletime}
Let $G_n$ be a random $r$-regular graph with $r\geq3$ and suppose that $0<h<C_0\sqrt{r}$ for some uniform constant $C_0>0$ small enough. Then, there exist uniform constants $0<C_1<\sqrt3/2$ and $C_2<\infty$ so that, whp,
\eq
\lim_{\beta\rightarrow\infty} \prob_{\boxminus}[e^{\beta(r/2-C_1\sqrt{r})n}< \tau_{\boxplus} < e^{\beta(r/2+C_2\sqrt{r})n}]=1.
\en
If $\eta$ is a metastable state, then the same holds with $\boxminus$ replaced by $\eta$.
If $r\geq6$, then $\boxminus$ is the unique metastable state.
\end{theorem}
The condition that $C_1<\sqrt3/2$ ensures that $r/2-C_1\sqrt{r} > 0$ for all $r\geq3$. The proof of this theorem can be found in the next section, where we verify Conditions {\bf (1)}, {\bf (2)},  {\bf (3a)} and, for $r\geq6$, {\bf (3b)} with
\eq
\Gamma_\ell = (r/2-C_1\sqrt{r})n \qquad {\rm and} \qquad \Gamma_u = (r/2+C_2\sqrt{r})n.
\en
Unfortunately, we were not able to prove that $\boxminus$ is the unique metastable state for all $r\geq3$, although we expect this to be the case. We show that $V_\sigma \leq C_3 \sqrt{r} n$ for $\sigma \notin \{\boxminus,\boxplus\}$ and need to show that this is less than $(r/2-C_1\sqrt{r})n$, but for small values of $r$ our bounds on the constants are not sharp enough and hence, we can only prove this for $r\geq6$. As will be shown below the proof of Proposition~\ref{prop-condition3b}, the main reason for this is that we cannot use that $i_e(G_n)>1$ for $r\leq5$.

The condition that $h<C\sqrt{r}$ is only necessary to prove that the metastable time behaves like $\exp(\beta (r/2+ \mathcal{O}(\sqrt{r}))n)$ for $r\rightarrow\infty$. For $h<i_e(G_n)$, the same proof strategy can be used, but this comes at the expense that the bounds on the metastable time become less tight.

The above result holds for random $r$-regular graphs. It would be interesting to generalize this result to random graphs with more general degree distributions. For this, it will be necessary to get a more detailed understanding of the structure of components that consist of a positive fraction of the graph and what measures will be of interest. For example, the isoperimetric number does not always give useful information, since this is equal to zero for disconnected graphs such as the Erd\H{o}s-R\'enyi random graph. 

Our main result shows that, in the zero temperature limit and for fixed $n$, the metastable time is exponentially big in $n$. It is expected, however, that the metastable time grows exponentially in $n$ for all temperatures below the critical temperature. In~\cite{SchShl98}, for example, a setting where the temperature is below the critical temperature, but strictly positive and the external field is vanishing is studied in two dimensions. It would be interesting to investigate this further for the model on the random graph. It would also be interesting to see if it is possible to let $n$ and $\beta$ tend to infinity simultaneously. 

It would be interesting to see if the pathwise approach can also be used to describe the other problems on metastability and gain more insight in the structure of the critical droplet and the tube of typical trajectories. It would also be interesting to see if the potential theoretic approach and martingale approach can yield improved results about the metastability of the Ising model on random graphs.


\section{Proofs}
In the next three subsections we verify Conditions {\bf (1)}, {\bf (2)},  {\bf (3a)} and, for $r\geq6$, {\bf (3b)} with
\eq
\Gamma_\ell = (r/2-C_1\sqrt{r})n \qquad {\rm and} \qquad \Gamma_u = (r/2+C_2\sqrt{r})n.
\en
Theorem~\ref{thm-metastabletime} then follows from Proposition~\ref{prop-metastable-time}, where it should be observed that the $\varepsilon$ can be absorbed into the constants.


\subsection{Condition (1): lower bound on communication height}
We start by deriving a lower bound on the communication height between $\boxminus$ and $\boxplus$, verifying Condition~{\bf (1)}:
\begin{proposition}[Lower bound on communication height]\label{prop-lb-commheight}
Let $G_n$ be a graph and suppose that $0<h<i_e(G_n)$. Then,
\eq\label{eq-comheightie}
\Phi(\boxminus,\boxplus) - H(\boxminus) \geq (i_e(G_n)-h)n.
\en
If $G_n$ is a random $r$-regular graph with $r\geq 3$ and $0<h<C_0\sqrt{r}$ for some uniform constant $0<C_0<(\sqrt{3}/2-\sqrt{\log(2)})$, then there exists a uniform constant $0<C_1<\sqrt{3}/2$ so that, whp,
\eq
\Phi(\boxminus,\boxplus) - H(\boxminus) \geq (r/2-C_1\sqrt{r})n.
\en
\end{proposition}
\begin{proof}
Recall that $G_n=(V_n=[n], E_n)$. Hence, for any subset $A\subseteq [n]$ with $|A|\leq n/2$, it holds that
\begin{align}
H(\sigma^A) &= -\sum_{(i,j)\in E_n} \sigma^A_i \sigma^A_j-h \sum_{i\in [n]} \sigma^A_i =  - (|E_n|-|\partial_e A|) +|\partial_e A|-h|A|+ h(n-|A|) \nn\\
&= 2 |\partial_e A|-2 h|A|-|E_n|+hn.
\end{align}
Hence, by the definition of the isoperimetric number,
\eq\label{eq-energybnd-isoperimetric}
H(\sigma^A) \geq 2 (i_e(G_n)-h)|A|-|E_n|+hn.
\en
Note that every path from $\boxminus$ to $\boxplus$ has to go through a configuration with $\lfloor n/2 \rfloor$ plus spins. Using the above, the energy of any such configuration  is at least $(i_e(G)-h)n-|E_n|+hn.$ The first statement of the proposition now follows by observing that
\eq
H(\boxminus) = -|E_n|+hn.
\en

To prove the second statement, observe that whp $i_e(G_n)\geq r/2-\sqrt{\log 2}\sqrt{r}$, so that whp $h<i_e(G_n)$ if $h<(\sqrt{3}/2-\sqrt{\log(2)})\sqrt{r}$ for $r\geq3$. Hence, it follows from~\eqref{eq-comheightie} that, whp,
\eq
\Phi(\boxminus,\boxplus) - H(\boxminus) \geq (i_e(G_n)-h)n \geq (r/2-C_1\sqrt{r}),
\en
with $C_1=\sqrt{\log 2}+C_0<\sqrt{3}/2$. 
\end{proof}
\begin{remark}
The first statement of the proposition above holds for general graphs and not only for $r$-regular graphs. For other graphs, however, it does not always give useful information. If, for example, the graph is not connected, which is, e.g., the case whp in the Erd\H{o}s-R\'enyi random graph, then  $i_e(G_n)=0$. Also for lattices $i_e(G_n)\rightarrow 0$ for $n\rightarrow \infty$, and hence the above bound is not useful.
\end{remark}


\subsection{Condition (2): upper bound on communication height}
In the following lemma we compute when the energy decreases if we flip one spin at vertex $i$ from $+$ to $-$ and vice versa.
\begin{lemma}\label{lem-changingplusspinsforfree}
Suppose that $G_n$ is an $r$-regular graph. Let $A\subseteq [n]$. Then, for all $i\in A$, $H(\sigma^{A \setminus i}) \leq H(\sigma^A)$ if and only if
$|\partial_e i \setminus A| \geq (r+h)/2$. 

Furthermore, for all $i\in A^c$, $H(\sigma^{A \cup i}) \leq H(\sigma^A)$ if and only if
$|\partial_e i \setminus A^c| \geq (r-h)/2$.
\end{lemma}
\begin{proof}
For $i\in A$, let $k_i^A = |\partial_e i \setminus A|$, the number of $-$ spins connected to vertex $i$ in configuration $\sigma^A$. Hence, $i$ is connected to $r-k_i^A$ $+$ spins. Thus, the difference in energy
\eq\label{eq-energydifonepllsuspinless}
H(\sigma^{A \setminus i}) - H(\sigma^A) = 2(r-k_i^A) - 2 k_i^A + 2h = - 4 k_i^A +2(r+h),
\en
since every edge connected to $i$ between two equal spins in $\sigma^A$ has one $+$ and one $-$ spin in $\sigma^{A \setminus i}$ and vice versa. This is indeed nonpositive iff $k_i^A \geq (r+h)/2$.

The proof for $i\in A^c$ is similar.
\end{proof}

\begin{remark}
If $(r+h)/2$ is noninteger, then $2(r+h)$ will not be divisible by $4$ and the energy difference by flipping one spin cannot be $0$, and hence under this extra condition the first statement of the lemma can be improved  to $H(\sigma^{A \setminus i}) < H(\sigma^A)$ if and only if $|\partial_e i \setminus A| \geq (r+h)/2$ and also in the second statement the analogous inequality can be made strict. This will however not lead to substantially better results in the rest of the paper, so we will not make this extra assumption.
\end{remark}

\begin{lemma}\label{lem-goingtolowerenergy1}
Suppose that $G_n$ is a connected $r$-regular graph and that $0<h<i_e(G_n)$. Then, for every set $A \subset [n]$ with $1\leq |A|\leq n/2$ there exists a set $B \subset A$ with $|B|<|A|$ such that
\eq
H(\sigma^B)<H(\sigma^A),
\en
and
\eq\label{eq-ub-PhisigmaAsigmaB}
\Phi(\sigma^A,\sigma^B)-H(\sigma^A) \leq 2 (r-2+h)s,
\en
with $s=\lceil\frac{r+h-2i_e(G_n)}{r-h}|A|\rceil$.

For $r\geq6$, the same holds, whp, with~\eqref{eq-ub-PhisigmaAsigmaB} replaced by
\eq\label{eq-ub-PhisigmaAsigmaB-rgeq6}
\Phi(\sigma^A,\sigma^B)-H(\sigma^A) \leq 2 (r-4+h)s.
\en
\end{lemma}
\begin{proof}
Given a set $A\subset[n]$, we construct the set $B$ as follows. We start from configuration $\sigma^A$ and then flip $s$ times a $+$ spin to $-$, where each step is a single spin-flip, choosing a random $+$ with at least one $-$ neighbor every step. If there is at least one $+$ spin left, such a $+$ spin always exists, since the graph is connected and $|A|\leq n/2$. If there are no $+$ spins left we are done, because it follows from~\eqref{eq-energybnd-isoperimetric} that $H(\boxminus)<H(\sigma^A)$. Every step the energy can go up at most $2(r-1)-2+2h = 2(r-2+h)$.

After changing these $s$ spins from $+$ to $-$,  we keep changing spins from $+$ to $-$, but now selecting a spin at random at every step that has at least $(r+h)/2$ $-$ neighbors. We continue doing this until such spins do not exist anymore. From Lemma~\ref{lem-changingplusspinsforfree} it follows that the energy can not go up in any of these these steps. We call the remaining configuration $\sigma^B$.

It is obvious that $|B|\leq|A|-s<|A|$. From the above it also follows that $\Phi(\sigma^A,\sigma^B)-H(\sigma^A) \leq 2(r-2+h)s$. It thus remains to choose $s$ big enough so that we are sure that $H(\sigma^B)<H(\sigma^A)$. For this, note that
\begin{align}
H(\sigma^B) &= 2 |\partial_e B| - 2h |B| -|E|+hn \nn\\
&< 2 \left(\frac{r+h}{2} -h\right)|B| -|E|+hn \leq (r-h)(|A|-s) -|E|+hn,
\end{align}
where in the first inequality we used that every vertex in $B$ has strictly less than $(r+h)/2$ $-$ neighbors by construction.

Hence, it follows from~\eqref{eq-energybnd-isoperimetric} that we need to choose $s$ such that
\eq
(r-h)(|A|-s) \leq 2 (i_e(G)-h)|A|,
\en
which is equivalent to
\eq
s \geq \frac{r+h-2i_e(G)}{r-h}|A|.
\en

As suggested to the author by Oliver Jovanovski, the bound~\eqref{eq-ub-PhisigmaAsigmaB} can be improved by using that $i_e(G_n)>1$. This is true, whp, for $r\geq6$ as mentioned in Proposition~\ref{prop-lb-isoperimetric}. As observed below Proposition~\ref{prop-lb-isoperimetric}, in this case there is always a $+$ spin that is connected to at least {\em two} $-$ neighbors. By choosing such a spin at random during the first $s$ spin flips, the energy can go up at most $2(r-2)-4+2h = 2(r-4+h)$ at every step.

The rest of the proof is analogous to the above. Note that the value of $s$ does not change, since this only makes use of the fact that in the end there are no longer $+$ spins with at least $(r+h)/2$ $-$ neighbors.
\end{proof}

\begin{remark}
In the above lemma we made essential use of the fact that all degrees are equal to get a bound on $|\partial_e B|$ in terms of $|B|$.
\end{remark}
We can make a similar statement if already more than half the spins are $+$.
\begin{lemma}\label{lem-goingtolowerenergy2}
Suppose that $G_n$ is a connected $r$-regular graph and that $h>0$. Then, for every set $A \subset [n]$ with $n/2 \leq |A|< n$ there exists a set $B \supset A$ with $|B|>|A|$ such that
\eq
H(\sigma^B)<H(\sigma^A),
\en
and
\eq\label{eq-ub-PhisigmaAsigmaB2}
\Phi(\sigma^A,\sigma^B)-H(\sigma^A) \leq 2 (r-2-h)s,
\en
with $s=\lceil\frac{r-h-2i_e(G_n)}{r+h}|A^c|\rceil$.
For $r\geq6$, the same holds, whp, with~\eqref{eq-ub-PhisigmaAsigmaB2} replaced by 
\eq\label{eq-ub-PhisigmaAsigmaB2-rgeq6}
\Phi(\sigma^A,\sigma^B)-H(\sigma^A) \leq 2 (r-4-h)s.
\en
\end{lemma}
The proof is similar to that of Lemma~\ref{lem-goingtolowerenergy1}. We can now verify Condition {\bf (2)}:

\begin{proposition}[Upper bound on communication height]\label{prop-ub-comm-height}
Suppose that $G_n$ is a connected $r$-regular graph, $0<h<i_e(G_n)$ and $n$ is large enough. Then,
\eq
\Phi(\boxminus,\boxplus)- H(\boxminus)  \leq (i_e'(G_n)-h) n + 2(r-2+h)\left\lceil\frac{r+h-2i_e(G_n)}{r-h}\frac{n}{2}\right\rceil. 
\en
If $G_n$ is a random $r$-regular graph and $0<h<C_0 \sqrt{r}$ for $C_0>0$ small enough, then, whp,
\eq
\Phi(\boxminus,\boxplus)- H(\boxminus)  \leq (r/2+C_2\sqrt{r})n,
\en
for some constant $C_2<\infty$.
\end{proposition}
\begin{proof}
Let $\mathcal{A}$ be a set of subsets satisfying
\eq
\mathcal{A}=\argmin_{\stackrel{A \subset [n]}{|A| = \lfloor n/2 \rfloor}} \frac{|\partial_e A|}{|A|},
\en
and let $A^*$ be an arbitrary element of $\mathcal{A}$. Then,
\eq
H(\sigma^{A^*})\leq (i_e'(G_n)-h) n -|E_n|+hn.
\en
From Lemma~\ref{lem-goingtolowerenergy1} it follows that there exists a set $B^*\subset A^*$ such that $|B^*|<|A^*|$ with $H(\sigma^{B^*})<H(\sigma^{A^*})$ and 
\eq
\Phi(\sigma^{A^*},\sigma^{B^*}) \leq H(\sigma^{A^*}) + 2 (r-2+h)\left\lceil\frac{r+h-2i_e(G_n)}{r-h}\frac{n}{2}\right\rceil.
\en
Now, it follows from Lemma~\ref{lem-goingtolowerenergy1} again that there exists a set $B'\subset B^*$ such that $|B'|<|B^*|$ with $H(\sigma^{B'})<H(\sigma^{B^*})$ and 
\begin{align}
\Phi(\sigma^{B^*},\sigma^{B'})&\leq H(\sigma^{B^*}) + 2 (r-2+h)\left\lceil\frac{r+h-2i_e(G_n)}{r-h}|B^*|\right\rceil\nn\\
&<H(\sigma^{A^*}) + 2 (r-2+h)\left\lceil\frac{r+h-2i_e(G_n)}{r-h}\frac{n}{2}\right\rceil.
\end{align}
We apply this recursively until we reach the set $\boxminus$. Hence,
\eq
\Phi(\boxminus,\sigma^{A^*}) -H(\boxminus)= \Phi(\sigma^A,\boxminus)-H(\boxminus) \leq (i_e'(G_n)-h) n + 2 (r-2+h)\left\lceil\frac{r+h-2i_e(G_n)}{r-h}\frac{n}{2}\right\rceil.
\en

To bound $\Phi(\sigma^{A^*},\boxplus)$ we first flip one $-$ spin in $\sigma^{A^*}$ that is connected to at least one $+$ to $+$ by which the energy can go up at most $2(r-2-h)$. If we call the resulting configuration $\sigma^{A'}$, then we are sure that $|A'|\geq n/2$, so that we can use Lemma~\ref{lem-goingtolowerenergy2} and similar arguments as above to show that 
\eq
\Phi(\sigma^{A'},\boxplus) -H(\boxminus) \leq (i_e'(G_n)-h) n + 2 (r-2-h)\left\lceil\frac{r-h-2i_e(G_n)}{r+h}\frac{n}{2}\right\rceil,
\en
and hence that
\eq
\Phi(\sigma^{A^*},\boxplus) -H(\boxminus) \leq (i_e'(G_n)-h) n + 2 (r-2-h)\left\lceil\frac{r-h-2i_e(G_n)}{r+h}\frac{n}{2}+1\right\rceil.
\en
Note that we can choose $n$ large enough so that $(r+h)/n<h$. Hence,
\eq\label{eq-compare-bds-gln2}
\frac{r-h-2i_e(G_n)}{r+h}\frac{n}{2}+1 = \frac{r-h+ 2/n(r+h)-2i_e(G_n)}{r+h}\frac{n}{2} \leq \frac{r+h-2i_e(G_n)}{r-h}\frac{n}{2},
\en
where we also used that $r+h>r-h$. The first statement of the proposition now follows by observing that
\eq
\Phi(\boxminus,\boxplus) -H(\boxminus) \leq \max\big(\Phi(\boxminus,\sigma^{A^*}),\Phi(\sigma^{A^*},\boxplus)\big) -H(\boxminus).
\en

To obtain the second result we use Propositions~\ref{prop-lb-isoperimetric} and~\ref{prop-ub-isoperimetric}. Let $C>0$ be a constant such that~\eqref{eq-ub-ieprime} holds. Then, whp,
\begin{align}
\Phi(\boxminus,\boxplus) -H(\boxminus)&\leq (r/2-C\sqrt{r}-h) n + 2 (r-2+h)\left\lceil\frac{r+h-2(r/2-\sqrt{\log 2}\sqrt{r})}{r-h}\frac{n}{2}\right\rceil.\nn\\
&\leq r/2 n +\frac{r+h}{r-h}\left(h+2\sqrt{\log 2}\sqrt{r}+2/n(r-h)\right)n.
\end{align}
If $h\leq C_0\sqrt{r}$ for some constant $C_0$ small enough, then also $h\leq r/2$, so that $(r+h)/(r-h)\leq3$. We can also choose $n$ large enough so that $2/n(r-h)<h$. 
Hence, whp,
\eq
\Phi(\boxminus,\boxplus) -H(\boxminus)\leq r/2 n +3\left(2C_0\sqrt{r}+2\sqrt{\log 2}\sqrt{r}\right)n
\leq (r/2+C_2\sqrt{r}) n,
\en
for some $C_2<\infty$.
\end{proof}
Note that for $r\geq6$ the value of $C_2$ can be improved a bit by using~\eqref{eq-ub-PhisigmaAsigmaB-rgeq6} and~\eqref{eq-ub-PhisigmaAsigmaB2-rgeq6} instead of~\eqref{eq-ub-PhisigmaAsigmaB} and~\eqref{eq-ub-PhisigmaAsigmaB2}, respectively.

\subsection{Conditions 3(a) and 3(b): upper bounds on stability levels}
It remains to verify Conditions~{\bf (3a)} and~{\bf (3b)}, which we do next, again using Lemma's~\ref{lem-goingtolowerenergy1} and~\ref{lem-goingtolowerenergy2}. We start with Condition~{\bf (3a)}:

\begin{proposition}[First upper bound on stability levels]
Suppose that $G_n$ is a connected $r$-regular graph and that $0<h<i_e(G_n)$. Then, for all $\sigma \notin \{\boxminus,\boxplus\}$,
\eq
V_\sigma \leq 2 (r-2+h)\left\lceil\frac{r+h-2i_e(G_n)}{r-h}\frac{n}{2}\right\rceil.
\en
If $G_n$ is a random $r$-regular graph with $r\geq3$ and $h<C_0 \sqrt{r}$ for $C_0>0$ small enough, then, whp, for all $\sigma \notin \{\boxminus,\boxplus\}$,
\eq
V_\sigma  \leq (r/2+C_2\sqrt{r})n,
\en
for some constant $C_2<\infty$.
\end{proposition}
\begin{proof}
The first inequality follows from Lemma~\ref{lem-goingtolowerenergy1} if $\sigma=\sigma^A$ with $|A|\leq n/2$ and from Lemma~\ref{lem-goingtolowerenergy2} and~\eqref{eq-compare-bds-gln2} if $\sigma=\sigma^A$ with $|A|\geq n/2$. 

Hence, it follows from Proposition~\ref{prop-lb-isoperimetric} that, whp,
\eq\label{eq-ub-V}
V_\sigma \leq 2 (r-2+h)\left\lceil\frac{r+h-2i_e(G_n)}{r-h}\frac{n}{2}\right\rceil \leq 2 (r-2+h)\left\lceil\frac{h+2\sqrt{\log 2}\sqrt{r}}{r-h}\frac{n}{2}\right\rceil.
\en
This can be bounded from above by $C_2\sqrt{r}n$ as in Proposition~\ref{prop-ub-comm-height}. Hence, whp,
\eq
V_\sigma \leq C_2\sqrt{r}n \leq (r/2+C_2\sqrt{r})n.
\en
\end{proof}
Unfortunately, the bound $V_\sigma \leq C_2\sqrt{r}n$ is not sufficient to prove that $V_\sigma < (r/2-C_1\sqrt{r})n$ for small values of $r$. Hence, we need to bound the constant $C_2$ more precisely to verify Condition~{\bf (3b)}. We do this in the next proposition for $r\geq6$:
\begin{proposition}[Second upper bound on stability levels]\label{prop-condition3b}
If $G_n$ is a random $r$-regular graph with $r\geq6$ and $h<C_0 \sqrt{r}$ for $C_0>0$ small enough, then, whp, for all $\sigma \notin \{\boxminus,\boxplus\}$,
\eq
V_\sigma < (r/2-C_1\sqrt{r})n.
\en
\end{proposition}
\begin{proof}
For $r\geq6$, using the improved bound in~\eqref{eq-ub-PhisigmaAsigmaB-rgeq6} of Lemma~\ref{lem-goingtolowerenergy1} we can improve~\eqref{eq-ub-V} to
\eq
V_\sigma \leq 2 (r-4+h)\left\lceil\frac{h+2\sqrt{\log 2}\sqrt{r}}{r-h}\frac{n}{2}\right\rceil \leq \frac{r-4+h}{r-h}(h+2\sqrt{\log 2}\sqrt{r} +2/n(r-h))n.
\en
We want to prove that we can choose $C_0$ small enough and $n$ big enough so that, for all $r\geq 6$,
\eq
V_\sigma < (r/2-C_1\sqrt{r})n,
\en
where $C_1 = \sqrt{\log 2}+C_0$ as observed in Proposition~\ref{prop-lb-commheight}. Hence, we need that  
\eq
\frac{r-4+h}{r-h}(h+2\sqrt{\log 2}\sqrt{r} +2/n(r-h)) + \sqrt{\log 2}+C_0 < r/2.
\en
We can rewrite this as 
\eq
2(1-4/r)\sqrt{\log 2}\sqrt{r}+\sqrt{\log 2}-r/2 + C_3 < 0,
\en
where $C_3$ can be chosen arbitrary small by choosing $C_0$ small enough and $n$ big enough. It is easy to check numerically that, for $r=6$ and $r=7$,
\eq\label{eq-conditionr}
2(1-4/r)\sqrt{\log 2}\sqrt{r}+\sqrt{\log 2}-r/2< 0.
\en
For $r\geq7$ it can easily be seen that the derivative of the l.h.s.\ of this equation is negative. Hence, it holds for all $r\geq6$.
\end{proof}
For $r\leq5$, we cannot use that $i_e(G_n)>1$ and hence we cannot use~\eqref{eq-ub-PhisigmaAsigmaB-rgeq6}. If we use~\eqref{eq-ub-PhisigmaAsigmaB} instead, condition~\eqref{eq-conditionr} becomes
\eq
2(1-2/r)\sqrt{\log 2}\sqrt{r}+\sqrt{\log 2}-r/2< 0,
\en
which can easily be checked not to hold for $r \leq 10$. Hence, for $r\leq5$ different ideas are necessary to show that $\boxminus$ is the unique metastable state.


\appendix
\section{Metastable time}\label{apx-metastabletime}
In this appendix, we prove a sequence of lemma's that together give the proof of Proposition~\ref{prop-metastable-time}. In fact, we show that the probabilities converge exponentially or even super-exponentially fast as $\beta\to\infty$.

We prove the lower bound on the hitting time of $\boxplus$ separately for the process starting in $\boxminus$ and in a metastable state $\eta$:
\begin{lemma}
If Condition {\bf (1)} holds, then, for all $\varepsilon>0$, $\delta\in(0,\varepsilon)$ and sufficiently large $\beta$,
\eq
\prob_\boxminus[\tau_\boxplus>e^{\beta (\Gamma_\ell-\varepsilon)}] \geq 1-e^{-\beta\delta}.
\en
\end{lemma}
\begin{proof}
To prove this lemma we need to introduce some terminology from~\cite{ManNarOliSco04}. For a non-empty set of configurations $\mathcal{A}$ denote by $\partial^{\rm ext}\mathcal{A}$ the external boundary of $\mathcal{A}$, i.e.,
\eq
\partial^{\rm ext}\mathcal{A} = \{ \sigma^B \notin \mathcal{A} : |A\triangle B|=1 {\rm\ for\ some\ }\sigma^A \in \mathcal{A}\}.
\en
We define a {\em non-trivial cycle} as a connected set of configurations $\mathcal{C}$ such that
\eq
\max_{\sigma \in \mathcal{C}} H(\sigma) < \min_{\eta \in \partial^{\rm ext}\mathcal{C}} H(\eta).
\en
The depth $D(\mathcal{C})$ of a non-trivial cycle $\mathcal{C}$ is defined as
\eq
D(\mathcal{C}) = \min_{\eta \in \partial^{\rm ext}\mathcal{C}} H(\eta) - \min_{\sigma \in \mathcal{C}} H(\sigma).
\en

We consider the cycle 
\eq
\mathcal{C} = \{\sigma : \Phi(\boxminus,\sigma) - H(\boxminus) < \Gamma_\ell\}.
\en
From the definition of the communication height it follows that for any $\sigma\in\mathcal{C}$ it holds that $H(\sigma) < \Gamma_\ell+H(\boxminus)$ and for any configuration $\eta \in \partial^{\rm ext}\mathcal{C}$ it holds that $\eta \notin \mathcal{C}$ and hence $H(\eta) \geq \Gamma_\ell+H(\boxminus)$. Hence, $\mathcal{C}$ is a non-trivial cycle.

We can thus use \cite[Theorem~2.17]{ManNarOliSco04} to conclude that, for any $\sigma\in \mathcal{C}$, $\varepsilon>0$, $\delta\in(0,\varepsilon)$ and sufficiently large $\beta$,
\eq
\prob_\sigma[\tau_{\partial^{\rm ext}\mathcal{C}}>e^{\beta (D-\varepsilon)}] \geq 1-e^{-\beta\delta},
\en
where
\eq
D= D(\mathcal{C}) = \min_{\eta \in \partial^{\rm ext}\mathcal{C}} H(\eta) - \min_{\sigma \in \mathcal{C}} H(\sigma) \geq \Gamma_\ell+H(\boxminus)-H(\boxminus)=\Gamma_\ell.
\en
Clearly, $\boxminus\in\mathcal{C}$, but $\boxplus\notin\mathcal{C}$ by Condition~{\bf (1)}. Hence if the process starts from $\boxminus$ then $\tau_\boxplus \geq \tau_{\partial^{\rm ext}\mathcal{C}}$. Hence,
\eq
\prob_\boxminus[\tau_{\boxplus}>e^{\beta (\Gamma_\ell-\varepsilon)}] \geq \prob_\boxminus[\tau_{\partial^{\rm ext}\mathcal{C}}>e^{\beta (\Gamma_\ell-\varepsilon)}] \geq \prob_\boxminus[\tau_{\partial^{\rm ext}\mathcal{C}}>e^{\beta (D-\varepsilon)}] \geq 1-e^{-\beta\delta}.
\en
\end{proof}

The lower bound on the metastable time for the process starting from a metastable state is given in the next lemma:
\begin{lemma}\label{lem-lbmetastabletime}
If $\eta$ is a metastable configuration and Condition {\bf (1)} holds, then
\eq
\Gamma = V_\eta \geq \Gamma_\ell,
\en
and for all $\varepsilon>0$, there exists a constant $K$ such that for all sufficiently large $\beta$
\eq
\prob_\eta[\tau_\boxplus>e^{\beta (\Gamma_\ell-\varepsilon)}] \geq 1-e^{-K\beta}.
\en
\end{lemma}
\begin{proof}
We prove this lemma using the same reasoning as in the proof of~\cite[Theorem~2.4]{CirNar13}. Suppose that Condition {\bf (1)} holds, and assume that $\Gamma<\Gamma_\ell$. Then it holds that $V_\sigma<\Gamma_\ell$ for all configurations $\sigma\neq\boxplus$. Hence, if we start with configuration $\sigma_0=\boxminus$, we can find a state $\sigma_1$ such that $H(\sigma_1)<H(\sigma_0)$ and $\Phi(\sigma_0,\sigma_1)-H(\sigma_0)<\Gamma_\ell$, i.e., there exists a path $\omega_0$ from $\sigma_0$ to $\sigma_1$ such that
\eq
\max_{\sigma'\in\omega_0} H(\sigma') <\Gamma_\ell + H(\sigma_0).
\en

As long as the new configuration with lower energy is not equal to $\boxplus$, we can repeat this argument and find configurations $\sigma_2, \sigma_3, \ldots, \sigma_n$ such that $H(\sigma_0)>H(\sigma_1)>H(\sigma_2)>\ldots>H(\sigma_n)$ and $\Phi(\sigma_i,\sigma_{i+1})-H(\sigma_i)<\Gamma_\ell$ for all $i=0,2,\ldots,n-1$, i.e, there exist paths $\omega_i$ so that
\eq
\max_{\sigma'\in\omega_i} H(\sigma') <\Gamma_\ell + H(\sigma_i).
\en

Since the number of configurations is finite and the energy is strictly decreasing every step, if we choose $n$ large enough we will end in the configuration $\sigma_n=\boxplus$. If we let $\omega$ be the concatenation of the paths $\omega_0,\omega_1,\ldots,\omega_n$, then $\omega$ is a path from $\boxminus$ to $\boxplus$ and
\eq
\max_{\sigma'\in\omega} H(\sigma') = \max_{0\leq i<n} \max_{\sigma'\in\omega_i} H(\sigma')<\Gamma_\ell + \max_{0\leq i<n}H(\sigma_i) = \Gamma_\ell + H(\boxminus).
\en
Hence,
\eq
\Phi(\boxminus,\boxplus) - H(\boxminus) < \Gamma_\ell,
\en
which is in contradiction with Condition {\bf (1)}.

So, if Condition {\bf (1)} holds, then $\Gamma\geq \Gamma_\ell$. Since, $\eta$ is assumed to be a metastable state $V_\eta=\Gamma\geq\Gamma_\ell$. The second statement of the lemma now immediately follows from~\cite[Theorem~4.1]{ManNarOliSco04}.
\end{proof}
 
The upper bound on the metastable time is proved in the next lemma. Here, a function $f(\beta)$ is called super-exponentially small (SES), denoted by $f(\beta)={\rm SES}$, if
\eq
\lim_{\beta\to\infty} \frac{1}{\beta}\log f(\beta)=-\infty.
\en

\begin{lemma}[Upper bound on the metastable time]
If Conditions {\bf (2)} and {\bf (3a)} hold, then, for all $\varepsilon>0$,
\eq
\sup_{\sigma} \prob_\sigma[\tau_\boxplus>e^{\beta (\Gamma_u+\varepsilon)}] = {\rm SES}.
\en
\end{lemma}
\begin{proof} 
Let
\eq
K_V=\{\sigma : V_\sigma > V\},
\en
be the so-called {\em metastable set at level V}. If we set $V=\Gamma_u$, then it follows from Conditions {\bf (2)} and {\bf (3a)} that $V_\sigma \leq V$ for all $\sigma \neq \boxplus$. Hence, $K_{\Gamma_u} = \{\boxplus\}$, since $V_\boxplus=\infty$. It thus follows from \cite[Theorem~3.1]{ManNarOliSco04} that, for all $\varepsilon>0$,
\eq
\sup_{\sigma}\prob_\sigma[\tau_\boxplus>e^{\beta (\Gamma_u+\varepsilon)}] = \textsc{SES}.
\en
\end{proof}
Because of the supremum, the same bound on the hitting time of $\boxplus$ holds for the process starting from any state $\sigma$ and in particular for the process starting in $\boxminus$ and in any metastable state $\eta$.

We finally characterize when $\boxminus$ is the unique metastable state:
\begin{lemma}
If Conditions {\bf (1)} and {\bf (3b)} hold, then $\boxminus$ is the unique metastable state.
\end{lemma}
\begin{proof}
We want to prove that $V_\boxminus\geq\Gamma_\ell$, since then
\eq
\{\boxminus\} = \argmax_{\sigma \neq \boxplus} V_\sigma.
\en
Suppose that on the contrary $V_\boxminus < \Gamma_\ell$. From Condition {\bf (3b)} it follows that also $V_\sigma<\Gamma_\ell$ for all $\sigma\notin\{\boxminus,\boxplus\}$. Then, by the same reasoning as in Lemma~\ref{lem-lbmetastabletime}, we can find a path $\omega$ from $\boxminus$ to $\boxplus$ such that
\eq
\max_{\sigma'\in\omega} H(\sigma') <\Gamma_\ell + H(\boxminus),
\en
and hence that
\eq
\Phi(\boxminus,\boxplus)- H(\boxminus) < \Gamma_\ell,
\en
which is in contradiction with Condition~{\bf (1)}. 
\end{proof}


\paragraph*{Acknowledgements.}
The author thanks Alessandra Bianchi and Francesca Collet for useful discussions on metastability. He also thanks Oliver Jovanovski for his suggestion how to improve Lemma~\ref{lem-goingtolowerenergy1} which also improved the main result and the anonymous referee for suggesting many useful improvements.
This research has been partially supported by Futuro in Ricerca 2010 (grant n. RBFR10N90W).



\begin{thebibliography}{99}
\bibliographystyle{plain}
\bibitem{Alo97}
N.~Alon.
\newblock On the edge-expansion of graphs.
\newblock {\em Combinatorics, Probability and Computing}, {\bf 6}(2):145--152, (1997).

\bibitem{BasDem12}
A.~Basak and A.~Dembo.
\newblock Ferromagnetic Ising measures on large locally tree-like graphs.
\newblock Preprint, arXiv:1205.4749, (2012).

\bibitem{BelLan10}
J.~Beltr\'an and C.~Landim.
\newblock Tunneling and metastability of continuous time Markov chains.
\newblock {\em Journal of Statistical Physics}, {\bf 140}(6):1065--1114, (2010).

\bibitem{BelLan11}
J.~Beltr\'an and C.~Landim.
\newblock Metastability of reversible finite state Markov processes.
\newblock {\em Stochastic Processes and their Applications}, {\bf 121}:1633-1677, (2011).

\bibitem{BelLan14}
J.~Beltr\'an and C.~Landim.
\newblock A martingale approach to metastability.
\newblock To appear in: {\em Probability Theory and Related Fields}, DOI: 10.1007/s00440-014-0549-9, (2014).

\bibitem{BenCer96}
G.~Ben~Arous and R.~Cerf.
\newblock Metastability of the three dimensional Ising model on a torus at very low temperatures.
\newblock {\em Electronic Journal of Probability}, {\bf 1}:10, (1996).

\bibitem{Bol88}
B.~Bollob\'as.
\newblock The isoperimetric number of random regular graphs.
\newblock {\em European Journal of Combinatorics}, {\bf 9}(3):241--244, (1988).

\bibitem{BovEckGayKle01}
A.~Bovier, M.~Eckhoff, V.~Gayrard and M.~Klein.
\newblock Metastability in stochastic dynamics of disordered mean-field models.
\newblock {\em Probability Theory and Related Fields}, {\bf 119}(1):99--161, (2001).

\bibitem{BovEckGayKle02}
A.~Bovier, M.~Eckhoff, V.~Gayrard and M.~Klein.
\newblock Metastability and low lying spectra in reversible Markov chains.
\newblock {\em Communications in Mathematical Physics}, {\bf 228}(2):219--255, (2002).

\bibitem{BovMan02}
A.~Bovier and F.~Manzo.
\newblock Metastability in Glauber dynamics in the low-temperature limit: beyond exponential asymptotics.
\newblock {\em Journal of Statistical Physics}, {\bf 107}(3--4):757--779, (2002).

\bibitem{CasGalOliVar84}
M.~Cassandro, A.~Galves, E.~Olivieri and M.E.~Vares.
\newblock Metastable behavior of stochastic dynamics: a pathwise approach.
\newblock {\em Journal of Statistical Physics}, {\bf 35}(5--6):603--634, (1984).

\bibitem{CheEtAl13}
H.~Chen, S.~Li, Z.~Hou, G.~He, F.~Huang and C.~Shen.
\newblock How does degree heterogeneity affect nucleation on complex networks?
\newblock {\em Journal of Statistical Mechanics: Theory and Experiment}, {\bf 2013}:P09014, (2013).

\bibitem{CheSheHouXin11}
H.~Chen, C.~Shen, Z.~Hou  and H.~Xin.
\newblock Nucleation in scale-free networks.
\newblock {\em Physical Review E}, {\bf 83}:031110, (2011).

\bibitem{CirNar13}
E.N.M.~Cirillo and F.R.~Nardi.
\newblock Relaxation height in energy landscapes: An application to multiple metastable states.
\newblock {\em Journal of Statistical Physics}, {\bf 150}(6):1080--1114, (2013).

\bibitem{SanGue08}
L.~De Sanctis and F.~Guerra.
\newblock Mean field dilute ferromagnet: high temperature and zero temperature behavior.
\newblock {\em Journal of Statistical Physics}, {\bf 132}:759--785, (2008).

\bibitem{DemMon10}
A.~Dembo and A.~Montanari.
\newblock Ising models on locally tree-like graphs.
\newblock {\em The Annals of Applied Probability}, {\bf 20}(2):565--592, (2010).

\bibitem{DomGiaHof10}
S.~Dommers, C.~Giardin\`a and R.~van~der Hofstad.
\newblock Ising models on power-law random graphs.
\newblock {\em Journal of Statistical Physics}, {\bf 141}(4):638--660, (2010).

\bibitem{DomGiaHof14}
S.~Dommers, C.~Giardin\`a and R.~van~der Hofstad.
\newblock Ising critical exponents on random trees and graphs.
\newblock {\em Communications in Mathematical Physics}, {\bf 328}(1):355--395, (2014).

\bibitem{HolNarOliSco03}
F.~den Hollander, F.R.~Nardi, E.~Olivieri and E.~Scoppola.
\newblock Droplet growth for three-dimensional Kawasaki dynamics.
\newblock {\em Probability Theory and Related Fields}, {\bf 125}:153--194, (2003).

\bibitem{HolOliSco00}
F.~den Hollander, E.~Olivieri and E.~Scoppola.
\newblock Metastability and nucleation for conservative dynamics.
\newblock {\em Journal of Mathematical Physics}, {\bf 41}(3):1424--1498, (2000).

\bibitem{Isi24}
E.~Ising.
\newblock Beitrag zur Theorie des Ferro- und Paramagnetismus.
\newblock PhD Thesis, University of Hamburg, (1924).

\bibitem{KotOli93}
R.~Koteck\'y and E.~Olivieri.
\newblock Droplet dynamics for asymmetric Ising model.
\newblock {\em Journal of Statistical Physics}, {\bf 70}(5):1121--1148, (1993).

\bibitem{Lig85}
T.M.~Liggett.
\newblock {\em Interacting Particle Systems}.
\newblock Springer, Berlin, (1985).

\bibitem{KotOli94}
R.~Koteck\'y and E.~Olivieri.
\newblock Shapes of growing droplets -- a model of escape from a metastable phase.
\newblock {\em Journal of Statistical Physics}, {\bf 75}(3):409--506, (1994).

\bibitem{LubSly14}
E.~Lubetzky and A.~Sly.
\newblock Universality of cutoff for the Ising model.
\newblock Preprint, arXiv:1407.1761, (2014).

\bibitem{ManNarOliSco04}
F.~Manzo, F.R.~Nardi, E.~Olivieri and E.~Scoppola.
\newblock On the essential features of metastability: tunnelling time and critical configurations.
\newblock {\em Journal of Statistical Physics}, {\bf 115}(1--2):591--642, (2004).

\bibitem{MonMosSly12}
A.~Montanari, E.~Mossel and A.~Sly.
\newblock The weak limit of Ising models on locally tree-like graphs.
\newblock {\em Probability Theory and Related Fields}, {\bf 152}:31--51, (2012).

\bibitem{MosSly09}
E.~Mossel and A.~Sly.
\newblock Rapid mixing of Gibbs sampling on graphs that are sparse on average.
\newblock {\em Random Structures and Algorithms}, {\bf 35}(2):250--270, (2009).

\bibitem{MosSly13}
E.~Mossel and A.~Sly.
\newblock Exact thresholds for Ising--Gibbs samplers on general graphs.
\newblock {\em The Annals of Probability}, {\bf 41}(1):294--328, (2013).

\bibitem{NevSch91}
E.J.~Neves and R.H.~Schonmann.
\newblock Critical droplets and metastability for a Glauber dynamics at very low temperatures.
\newblock {\em Communications in Mathematical Physics}, {\bf 137}(2):209--230, (1991).

\bibitem{New03}
M.E.J.~Newman.
\newblock The structure and function of complex networks.
\newblock {\em SIAM Review}, {\bf 45}(2):167--256, (2003).

\bibitem{Nis05}
M.~Niss.
\newblock History of the Lenz--Ising model 1920--1950: from ferromagnetic to cooperative phenomena.
\newblock {\em Archive for History of Exact Sciences}, {\bf 59}(3):267--318, (2005).

\bibitem{Nis09}
M.~Niss.
\newblock History of the Lenz--Ising Model 1950--1965: from irrelevance to relevance.
\newblock {\em Archive for History of Exact Sciences}, {\bf 63}(3):243--287, (2009).

\bibitem{Nis11}
M.~Niss.
\newblock History of the Lenz--Ising Model 1965--1971: the role of a simple model in understanding critical phenomena.
\newblock {\em Archive for History of Exact Sciences}, {\bf 65}(6):625--658, (2011).

\bibitem{OliSco95}
E.~Olivieri and E.~Scoppola.
\newblock Markov chains with exponentially small transition probabilities: first exit problem from a general domain. I. The reversible case.
\newblock {\em Journal of Statistical Physics}, {\bf 79}(3):613--647, (1995).

\bibitem{OliVar05}
E.~Olivieri and M.E.~Vares.
\newblock {\em Large Deviations and Metastability.}
\newblock Cambridge University Press, Cambridge, (2005).

\bibitem{Sch92}
R.H.~Schonmann.
\newblock The pattern of escape from metastability of a stochastic Ising model.
\newblock {\em Communications in Mathematical Physics}, {\bf 147}(2):231--240, (1992).

\bibitem{SchShl98}
R.H.~Schonmann and S.B.~Shlosman.
\newblock Wulff droplets and the metastable relaxation of kinetic Ising models.
\newblock {\em Communications in Mathematical Physics}, {\bf 192}:389--462, (1998).

\bibitem{SheCheYeHou13}
C.~Shen, H.~Chen, M.~Ye and Z.~Hou.
\newblock Nucleation pathways on complex networks.
\newblock {\em Chaos}, {\bf 23}:013112, (2013).

\end{thebibliography}
\end{document}